\documentclass[12pt,leqno]{amsart}
\usepackage{amsmath,amssymb,txfonts}
\usepackage{amssymb}
\usepackage{amsxtra}
\usepackage{amsmath}
\usepackage{txfonts}
\usepackage{color}
\usepackage[cp1252]{inputenc}
\usepackage{mathrsfs}
\usepackage{graphicx}

\usepackage{soul} 

\allowdisplaybreaks

\newtheorem{theorem}{Theorem}[section]
\newtheorem{lemma}{Lemma}[section]

\newtheorem{remark}{Remark}[section]

\def\rr{{\mathbb R}}
\def\rn{{{\rr}^n}}

\def\cf{{\mathcal F}}

\def\kz{\kappa}

\def\fint{{\ifinner\rlap{\bf\kern.25em--}
\int\else\rlap{\bf\kern.45em--}\int\fi}\ignorespaces}

\def\diam{{\mathop\mathrm{\,diam\,}}}

\begin{document}

\title{A $(\phi_n, \phi)$-Poincar\'e inequality on John domain }

\begin{abstract}
Given a bounded domain  $\Omega \subset \rn$ with $n\ge2$,
let   $\phi $  be  a  Young function satisfying
the doubling condition with  the   constant  $K_\phi<2^{n}$.
 If $\Omega$ is a John domain,
we show that $\Omega $ supports a
$(\phi_{n}, \phi)$-Poincar\'e   inequality.
 Conversely, if $\Omega$ supports a $(\phi_n, \phi)$-Poincar\'e   inequality with an additional assumption that it is a simply connected domain when $n=2$ or a bounded domain which is quasiconformally equivalent to some uniform domain when $n\ge3$, then it is a John domain.
\end{abstract}


\author{Shangying Feng and Tian Liang}

\address{Shangying Feng\\ Department of Mathematics\\ Beijing Normal University\\ Beijing 100191, P.R. China;}
\email{\tt 202021130044@mail.bnu.edu.cn}

\address{Tian Liang\\
School of Mathematics and Statistics\\
Huizhou University\\
Guangdong 516007\\
P.R. China}
\email{\tt liangtian@hzu.edu.cn}

\date{\today }
\maketitle

\section{Introduction}\label{s1}

Orlicz-Sobolev spaces are particularly useful when the non-linearity in the problem does not follow a polynomial growth, which is a common assumption in classical Lebesgue spaces. 
The non-polynomial nature of the Orlicz function captures a wide range of nonlinear behaviors, making them suitable for modeling real-world phenomena with non-standard growth conditions.
The cited references, such as  \cite{ am02, BS12,  BS11, t90, w10} , likely delve into the theory of Orlicz spaces, their connections to Sobolev embeddings, and the specific results and techniques developed for solving problems in these spaces. 

Let  $\Omega$ be a bounded domain in $\rn$ with $n\ge2$.
Assume $\phi$ is a Young function in $[0, \infty)$,  that is,  
$\phi \in C[0, \infty)$ is convex and satisfies $\phi(0)=0,  \phi(t)>0$ for $t>0$ 
and $\lim \limits_{t \to \infty}\phi(t)=\infty$.
Recall the Orlicz space $L^{\phi}(\Omega)$ as the collection of all measurable functions $u$ in $\Omega$ with the semi-norm
\begin{eqnarray*}
\|u\|_{{L}^{ \phi}(\Omega)} := \inf \left\{\lambda > 0 : \int_{\Omega}  \phi \left(\frac{|u(x)|}{\lambda}\right) dx\leq 1 \right\}<\infty.
\end{eqnarray*}
The classical Orlicz-Sobolev space $W^{1, \phi}(\Omega)$ consists of all measurable functions $u \in L^{\phi}(\Omega)$ and $\nabla u \in L^{\phi}(\Omega)$, whose norm is
\begin{align*}
\|u\|_{W^{1, \phi}(\Omega)}:=\|u\|_{L^\phi(\Omega)}+\|\nabla u\|_{L^\phi(\Omega)}.
\end{align*}
Sometimes we consider the homogeneous Orlicz-Sobolev space $\dot{W}^{1, \phi}(\Omega)$ with its norm
$\|u\|_{W^{1, \phi}(\Omega)}:=\|\nabla u\|_{L^\phi(\Omega)}$,
whose  sharp embedding  has been solved in \cite{c96}(see also \cite{c97} for an alternate formulation of the solution)
and in \cite{refc1} for high order version.
The detailed description is as follows.

\begin{theorem}
Let  $\Omega$ is an open bounded domain in $\rn$ with finite measure and $\phi$ is a Young function  satisfying
\begin{equation}\label{eq2}
\int_0^t\left(\frac{\tau}{\phi(\tau)}\right)^\frac{1}{n-1}d\tau<\infty.
\end{equation}
Define    $\phi_n:=\phi\circ H^{-1}$,
where
\begin{equation}\label{eq5}
H(t)=\left[\int_0^t\left(\frac{\tau}{\phi(\tau)}\right)^\frac{1}{n-1}d\tau\right]^\frac{n-1}{n}~\forall ~t\geq0.
\end{equation}
Then $\dot{W}^{1, \phi}(\Omega)\subset L^{\phi_{n}}(\Omega)$, that is,
for any $u\in \dot{W}^{1, \phi}(\Omega)$, one has $u\in  L^{\phi_{n}}(\Omega)$ with
$\|u\|_{L^{\phi_{n}}(\Omega)}\le C \|u\|_{\dot{W}^{1, \phi}(\Omega)}$, where $C$ is a constant independent of $u$.
\end{theorem}

We are interested in bounded domains which supports the imbedding
$\dot{W}^{1, \phi}(\Omega)\subset L^{\phi_{n}}(\Omega)$ or  $(\phi_n, \phi)$-Poincar\'e inequality,
that is, there exists a constant $C\geq1$ such that
\begin{align}\label{ueq1}
\|u-u_\Omega\|_{L^{\phi_n}(\Omega)}\leq C\|u\|_{\dot{W}^{1, \phi}(\Omega)},
\end{align}
where $u_{\Omega}=  \fint_{\Omega}u = \frac{1}{|\Omega|} \int _{\Omega} u dx $
denotes the average of $u$ in the set of $\Omega$ with $\lvert \Omega\rvert>0$.

The primary goal of this paper is to effectively characterizes
supports  the imbedding   $\dot{W}^{1, \phi}(\Omega) \subset L^{\phi_{n}}(\Omega)$
via John domains under certain doubling assumption in $\phi$;
see   Theorem 1.2 below.
Recall that a bounded domain $\Omega \subset \rn$ is called as a $c$-John domain with respect to some $x_0 \in \Omega $ for some $c>0$ if 
for  any $x \in\Omega$,  
there is a rectifiable curve $\gamma : [0, T] \rightarrow \Omega$ parameterized by arc-length such that $\gamma(0)=x$,  $\gamma(T)=x_0$ and $d(\gamma(t),  \Omega ^{\complement}) > ct $ for all $t>0$.
We could refer to  \cite{ b82, bk95, bsk96, bkl95,m79, m88, r83}
and references therein for more study on $c$-John domains.
Moreover, a Young function $\phi$ has the doubling property ($\phi \in \Delta_2$) if
\begin{equation}\label{da2}
K_{\phi}:=\sup_{x>0}\frac{\phi(2t)}{\phi(t)}<\infty.
\end{equation}
It is well known that if a Young function $\phi \in \Delta_2$ with $K_\phi<2^{n}$,  
then $\phi$ satisfies \eqref{eq2},  see Lemma \ref{s2.l2}.

\begin{theorem}\label{th1}
 Let $\phi$ be a Young function and $\phi \in \Delta_2$ with $K_\phi<2^{n}$ in \eqref{da2}.
\begin{enumerate}
\item[(i)] If $\Omega \subset \mathbb{R}^n $ is a $c$-John domain,
then $\Omega$ supports the $(\phi_n, \phi)$-Poincar\'e inequality \eqref{ueq1} with the constant $C$ depending on $n,    c$ and $K_\phi$.

\item[(ii)] Assume that $\Omega \subset \mathbb{R}^n $ is a bounded simply connected planar domain,  or a bounded domain which is a quasiconformally equivalent to some uniform domain when $n\geq3$. If $\Omega$ supports the $(\phi_n, \phi)$- Poincar\'e inequality,  then $\Omega$ is a c-John domain,  where the constant $c$ depend on $n,    C, K_\phi$ and $\Omega$.
\end{enumerate}
\end{theorem}

\begin{remark}\rm
(i) For special Young function $\Phi$, 
then $(\phi_n, \phi)$-Poincar\'e inequality \eqref{ueq1} is the classical Sobolev $\dot{W}^{1,p}$-imbedding.
Letting $\Phi=t^p$ for some $p \in[1, n)$,
we know $L^\phi(\Omega)=L^p(\Omega)$ and $\phi_{n}(t)=C t^{\frac{np}{n-p}}$.
Then $(\phi_n, \phi)$-Poincar\'e inequality \eqref{ueq1} equals to
Sobolev $\dot{W}^{1,p}$-imbedding or $(\frac{np}{n-p},p)$-Poincar\'e inequality. Namely, 
for any  $u \in \dot{ W }^{1, p}(\Omega)$, there exists a constant $C>0$ such that
\begin{align}\label{1p}
\|u- u_{\Omega}\|_{L^{np/(n-p)}(\Omega)} \leq C \|u\|_{\dot{ W }^{1, p}(\Omega)},
\end{align}
where the constant $C$ depends on $n,p$ and $c$.
Noted that $c$-John domain  $\Omega$ supports $(\frac{np}{n-p},p)$-Poincar\'e inequality,
details see Reshetnyak \cite{r83} and Martio \cite{m88} for $1<p<n$ and
Borjarski \cite{b89} (and also Hajlasz \cite{h01}) for $p=1$.
On the other hand,  additionally assume that $\Omega$ is a
bounded simply connected planar domain or a domain that is quasiconformally equaivalently to
some uniform domain when $n\ge3$,
Buckley and Koskela \cite{bk95} proved that if   \eqref{1p} holds,
then $\Omega$ is a $c$-John domain.
The Haj\l sz version see  \cite{z11}.
The results are the same as Theorem \ref{th1}.

(ii) This theorem describes the geometric features of the embeddings more precisely than Heikkinen and Tuominen \cite{refh1}.
They proved that if $\Phi \in \Delta_2$ 
with $\frac{\Phi(t)}{t^{p}}$ decreasing for $1\leq p< n$ and $\dot{W}^{1, \phi}(\Omega) \subset L^{\phi_{n}}(\Omega)$, then $\Omega$ satisfies the measure density condition. 
Actually, when $\frac{\Phi(t)}{t^{p}}$ is decreasing for $1\leq p< n$, then $\frac{\Phi(2t)}{(2t)^{p}}\leq \frac{\Phi(t)}{t^{p}}$ for all $t>0$. 
In other words,  $\frac{\Phi(2t)}{(\Phi(t)}\leq 2^{p}<2^{n}$.
It means that $K_\phi<2^{n}$. 
Moreover, from Buckley and Koskela \cite{bsk96}, we knew that a John domain satisfies the measure density condition.
Hence in some terms, we extend the result that we characterize the  geometric features of the $(\phi_n, \phi)$-Poincar\'e inequality .

(iii) John domain characterize many embeddings, such us the intrinsic fractional Sobolev embedding $\dot{W}_{*}^{s,p}(\Omega ) \subset L^{np/n-sp}(\Omega )$ for $s \in (0,1)$ in \cite{h13}, 
the Haj\l sz version of fractional Sobolev embedding in \cite{z123},
the Orlicz-Besov embedding in \cite{refsh1},  the fractional Orlicz-Sobolev embedding in \cite{reflt3} and so on.
\end{remark}

The paper is organized as follows.
The proof of  Theorem \ref{th1}(i)  proven in Section 2  is similarly to \cite{reflt3} by the different embedding $\dot{W}^{1, \phi}(Q)\subset L^{\phi_{n}}(Q)$ for cube $Q$ and the vector-valued inequality in Orlicz norms for the Hardy-Littlewood maximum operators. 
Section 2 also contains some property of the doubling Young function.
Conversely,   together with the aid of some ideas from \cite{bk95,  refh1,reflt3, refsh1, z123},  
we obtain the $LLC(2)$ property of $\Omega$,  and then
 prove Theorem \ref{th1}(ii) by a capacity argument;
see Section 3 for details.

\section{Proof of Theorem 1.2(i)}\label{s2}
First we give the embedding $C_c^\infty(\Omega) \subset\dot{W}^{1, \phi}(\Omega)$, 
which means that $\dot{W}^{1, \phi}(\Omega)$ contains basic functions.
In some terms,  $\dot{W}^{1, \phi}(\Omega)$ is useful.
\begin{lemma}\label{s2.l1}
Let $\phi$ be a  Young function. For any bounded domain $\Omega \subset \mathbb{R}^n$,  we have $C_c^\infty(\Omega) \subset \dot{W}^{1, \phi}(\Omega)$.
\end{lemma}

\begin{proof}
Write $L:=\|Du\|_{L^\infty(\Omega)}$ and
choose $\mathrm{supp }u \subset W\subset\Omega$ such that $|\nabla u(x)|=0$ for $x\in \Omega \backslash W$.

For any $ u\in C_c^1(\Omega) $, we know
\begin{align*}
H:=\int_\Omega\phi\left(\frac{|\nabla u(x)|}{\lambda}\right)dx
\leq\int_W\phi\left(\frac{L}{\lambda }\right)dx
=\phi\left(\frac{L}{\lambda }\right)|W|.
\end{align*}
If $\lambda=(L+1)\left({\phi}^{-1}\left(\frac{1}{|W|}\right)\right)^{-1},$ we have $H\leq 1$.
Hence $u\in\dot{W}^{1,\phi}(\Omega)$.
Moreover, we obtain $C_c^\infty(\Omega) \subset  C_c^1(\Omega)$, which is the desired result.
\end{proof}

Now we give some lemmas  of Young function $\phi$ with the doubling property.
\begin{lemma}\label{s2.l2}
 Let $\phi \in \Delta_2$ be a Young function satisfying  $K_\phi<2^{n}$, then $\phi$ satisfy \eqref{eq2}.
\end{lemma}

\begin{proof}
Since $\phi \in \Delta_2$ with $\phi(2t)\leq K_\phi\phi(t)$,  we have
\begin{align*}
\int_\frac{t}{2}^t\left(\frac{\tau}{\phi(\tau)}\right)^\frac{1}{n-1}d\tau
=2\int_\frac{t}{4}^\frac{t}{2}\left(\frac{2\tau}{\phi(2\tau)}\right)^\frac{1}{n-1}d\tau
\geq 2 \int_\frac{t}{4}^\frac{t}{2}\left(\frac{2\tau}{K_\phi\phi(\tau)}\right)^\frac{1}{n-1}d\tau,
\end{align*}
that is,
\begin{align*}
\int_\frac{t}{4}^\frac{t}{2}\left(\frac{\tau}{\phi(\tau)}\right)^\frac{1}{n-1}d\tau
\leq\frac{K_\phi^\frac{1}{n-1}}{2^\frac{n}{n-1}}
\int_\frac{t}{2}^t\left(\frac{\tau}{\phi(\tau)}\right)^\frac{1}{n-1}d\tau.
\end{align*}
Therefore,
\begin{align*}
\int_\frac{t}{2^m}^\frac{t}{2^{m-1}}\left(\frac{\tau}{\phi(\tau)}\right)^\frac{1}{n-1}d\tau
\leq\frac{K_\phi^\frac{1}{n-1}}{2^\frac{n}{n-1}}\int_\frac{t}{2^{m-1}}^\frac{t}{2^{m-2}}
\left(\frac{\tau}{\phi(\tau)}\right)^\frac{1}{n-1}d\tau
\leq\left(\frac{K_\phi^\frac{1}{n-1}}{2^\frac{n}{n-1}}\right)^{m-1}
\int_\frac{t}{2}^t\left(\frac{\tau}{\phi(\tau)}\right)^\frac{1}{n-1}d\tau.
\end{align*}
Using $K_\phi<2^{n}$, there exists a constant $C>0$ such that $\sum\limits_{m=1}^\infty \left(\frac{K_\phi^\frac{1}{n-1}}{2^\frac{n}{n-1}}\right)^{m-1}\leq C$.
Hence
\begin{align*}
\int_0^t\left(\frac{\tau}{\phi(\tau)}\right)^\frac{1}{n-1}d\tau\leq\sum_{m=1}^\infty \left(\frac{K_\phi^\frac{1}{n-1}}{2^\frac{n}{n-1}}\right)^{m-1}
\int_\frac{t}{2}^t\left(\frac{\tau}{\phi(\tau)}\right)^\frac{1}{n-1}d\tau
\leq C\int_\frac{t}{2}^t\left(\frac{\tau}{\phi(\tau)}\right)^\frac{1}{n-1}d\tau.
\end{align*}
On the other hand, because $\frac{t}{\phi(t)}$ is decreasing, we have
\begin{align*}
\int_\frac{t}{2}^t\left(\frac{\tau}{\phi(\tau)}\right)^\frac{1}{n-1}d\tau
\leq\left(\frac{\frac{t}{2}}{\phi(\frac{t}{2})}\right)^\frac{1}{n-1}\frac{t}{2}<\infty.
\end{align*}
As $t>0$ could be any positive number, we conclude
\begin{align*}
\int_0^t\left(\frac{\tau}{\phi(\tau)}\right)^\frac{1}{n-1}d\tau
\leq C\left(\frac{\frac{t}{2}}{\phi(\frac{t}{2})}\right)^\frac{1}{n-1}\frac{t}{2}<\infty.
\end{align*}
\end{proof}

\begin{lemma}\label{s2.l3}
 Let $\phi \in \Delta_2$ be a Young function satisfying  $K_\phi<2^{n}$. Then there exists a constant
 $C>0$ such that
\begin{equation}\label{te1}
\frac{H(A)}{A}\leq \frac{C }{{\phi(A)}^\frac{1}{n}}.
\end{equation}
\end{lemma}
\begin{proof}
Applying Lemma \ref{s2.l2},
\begin{equation*}
\int_0^t\left(\frac{\tau}{\phi(\tau)}\right)^\frac{1}{n-1}d\tau
\leq C\int_\frac{t}{2}^t\left(\frac{\tau}{\phi(\tau)}\right)^\frac{1}{n-1}d\tau
\leq C\left(\frac{\frac{t}{2}}{\phi(\frac{t}{2})}\right)^\frac{1}{n-1}\frac{t}{2}.
\end{equation*}
Together with $K_\phi<2^{n}$, we have
\begin{align*}
H(A)
&=\left[\int_0^A\left(\frac{\tau}{\phi(\tau)}\right)^\frac{1}{n-1}d\tau\right]^\frac{n-1}{n}
\leq \left[C\left(\frac{\frac{A}{2}}{\phi(\frac{A}{2})}\right)^\frac{1}{n-1}\frac{A}{2}\right]
^\frac{n-1}{n}
\leq
\left[C\left(\frac{\frac{A}{2}}{\frac{1}{K_\phi}\phi(A)}\right)^\frac{1}{n-1}\frac{A}{2}\right]
^\frac{n-1}{n}
\leq \frac{CA}{{\phi(A)}^\frac{1}{n}}.
\end{align*}
Hence we get
\begin{align*}
\frac{H(A)}{A}
\leq\frac{C}{{\phi(A)}^\frac{1}{n}}
\end{align*}
as desired.
\end{proof}

The Young function $\phi$ is in  $\nabla_{2}$ ($\phi \in \nabla_2$) if there exists a  constant $a>1$  such that
\begin{equation*}
\phi(x)\leq\frac{1}{2a}\phi(ax), ~\forall x\geq0.
\end{equation*}
\begin{lemma}\label{s2.l4}
If $\phi \in \Delta_2$ be a Young function satisfying $K_{\phi}<2^{n}$, then $\phi_n \in \Delta_2\cap\nabla_2$.
\end{lemma}
\begin{proof}
Write
\begin{align*}
H(2t)=\left[\int_0^{2t}\left(\frac{\tau}{\phi(\tau)}\right)^\frac{1}{n-1}d\tau\right]^\frac{n-1}{n}
\geq\left[\int_0^t\left(\frac{2\tau}{K_{\phi}\phi(\tau)}\right)^\frac{1}{n-1}2d\tau\right]^\frac{n-1}{n}
=\frac{2}{K_{\phi}^\frac{1}{n}}H(t).
\end{align*}
Letting $2y=H(2t)$,  we know $K_{\phi}^\frac{1}{n}y\geq H\left(\frac{H^{-1}(2y)}{2}\right)$.
Therefore,
$$H^{-1}(2y)\leq2H^{-1}(K_{\phi}^\frac{1}{n}y)\leq2^2H^{-1}\left(K_{\phi}^\frac{1}{n}\frac {K_{\phi}^\frac{1}{n}}{2}y\right)\leq...\leq2^{m+1}H^{-1}\left(K_{\phi}^\frac{1}{n}\left(\frac {K_{\phi}^\frac{1}{n}}{2}\right)^my\right).$$
Because of the range of K,  we get $\frac{K_{\phi}^\frac{1}{n}}{2}<1$.
Putting $m$  big enough so that
$K_{\phi}^\frac{1}{n}\left(\frac {K_{\phi}^\frac{1}{n}}{2}\right)^m<1$,
 we have $H^{-1}(2y)<CH^{-1}(y)$.
Hence  $H^{-1}\in \Delta_2$  and $\phi_n=\phi \circ H^{-1}\in \Delta_2$.

By the decreasing property of $\frac{\tau}{\phi(\tau)}$,
\begin{align*}
H(2^nx)&=\left(\int_0^{2^n x}\left(\frac{\tau}{\phi(\tau)}\right)^\frac{1}{n-1}d\tau\right)
^\frac{n-1}{n}
=\left(\int_0^x\left(\frac{2^n\tau}{\phi(2^n\tau)}\right)^\frac{1}{n-1}
2^n d\tau
\right)^\frac{n-1}{n}\\
&\leq\left(\int_0^x\left(\frac{\tau}{\phi(\tau)}\right)^\frac{1}{n-1}2^n d\tau\right)^\frac{n-1}{n}
=2^{n-1}H(x).
\end{align*}
Hence $2^nx\leq H^{-1}(2^{n-1}H(x))$, it means that $2^n H^{-1}(x)\leq H^{-1}(2^{n-1}x)$.
Moreover,
 \begin{align*}
2^n\phi \circ H^{-1}(x)\leq \phi(2^n H^{-1}(x))\leq\phi \circ H^{-1}(2^{n-1}x).
\end{align*}
Letting $a=2^{n-1}>1$, we have $\phi_n(x)\leq\frac{1}{2a}\phi_n(ax)$ and $\phi_n\in \nabla_2$.
\end{proof}

To prove Theorem \ref{th1}(i), we also need the following result to get the estimate of the embedding of cube.
\begin{lemma}\label{s2.l5}
Let $\phi$ be a Young function satisfying \eqref{eq2}.
Then for any cube $Q \subset \rn$,
$u \in {\dot {W}}^{1,\phi}(Q)$ and $ \lambda\geq C_1\|\nabla u\|_{L^\phi(Q)}$,
there exists a constant $C_1=C_1(n)>0$ such that
\begin{equation}\label{eq2.1}
\int_Q \phi_n\left(\frac{|u(x)-u_Q|}{\lambda}\right)dx\leq\int_Q \phi\left(\frac{C_1|\nabla u(x)|}{\lambda}\right)dx.
\end{equation}
\end{lemma}

\begin{proof}
Let $m=1,A=\phi$ and $\Omega=Q(O,1)$,
where $$Q(a,b)=\left\{(x_1,x_2,...,x_n)\in\mathbb{R}^n:|x_i-a_i|<b, a=(a_1,a_2,...,a_n)\in\mathbb{R}^n,b>0\right\}.$$
By theorem 4.3 in \cite{refc1}, there exists constants $C_1$ and $c$ such that
\begin{align*}
\int_{Q(0,1)}\phi_n\left(\frac{|u(x)-c|}{C_1\left(\int_{Q(0,1)}\phi(|\nabla u(x)|)dx\right)^\frac{1}{n}}\right)dx
\leq\int_{Q(0,1)}\phi\left(|\nabla u(x)|\right)dx.
\end{align*}
In fact, the above inequality holds if $c=u_{Q(0,1)}$.

If a cube centered at $a$ with sides of length $2l$ paralleled to the axes,
there exists an orthogonal transformation $T$ such that $T(Q-a)=Q(O,l)$.
For any $ u \in {\dot {W}}^{1,\phi}(Q)$,
 put $v(x)=\frac{u(T^{-1}(lx)+a)}{\lambda l}$ where $x\in Q(0,1)$.
Then $v\in {\dot {W}}^{1,\phi}(Q(0,1))$.
Hence
\begin{align*}
\int_{Q(0,1)}\phi_n\left(\frac{|\frac{u(T^{-1}(lx)+a)}{\lambda l}-c|}{C_1\left[\int_{Q(0,1)}\phi\left(l\left|\nabla\left(\frac{u(T^{-1}(lx)+a)}{\lambda l}\right)\right|\right)dx\right]^\frac{1}{n}}\right)dx
\leq\int_{Q(0,1)}\phi\left(l\left|\nabla\left(\frac{u(T^{-1}(lx)+a)}{\lambda l}\right)\right|\right)dx.
\end{align*}
Using $y=T^{-1}(lx)+a$, we know
\begin{align*}
\int_{Q}\phi_n\left(\frac{|\frac{u(y)}{\lambda l}-c|}{C_1\left[\int_{Q}\phi\left(l\left|\nabla\left(\frac{u(y)}{\lambda l}\right)\right|\right)\frac{dy}{l^n}\right]^\frac{1}{n}}\right)\frac{dy}{l^n}
\leq\int_{Q}\phi\left(l\left|\nabla\left(\frac{u(y)}{\lambda l}\right)\right|\right)\frac{dy}{l^n},
\end{align*}
that is,
\begin{align*}
\int_{Q}\phi_n\left(\frac{|u(y)-\lambda lc|}{C_1\lambda\left(\int_{Q}\phi\left(\frac{|\nabla u(y)|}{\lambda}\right)dy\right)^\frac{1}{n}}\right)dy
\leq\int_{Q}\phi\left(\frac{|\nabla u(y)|}{\lambda}\right)dy.
\end{align*}
Since $c=v_{Q(O,1)}$, we get $\lambda lc=u_Q$.
By variable substitution $u=C_1u$, we have
\begin{align*}
\int_{Q}\phi_n\left(\frac{|u(y)-u_Q|}{\lambda(\int_{Q}\phi\left(\frac{C_1|\nabla u(y)|}{\lambda}\right)dy)^\frac{1}{n}}\right)dy
\leq\int_{Q}\phi\left(\frac{C_1|\nabla u(y)|}{\lambda}\right)dy.
\end{align*}
If $\lambda\geq C_1\|\nabla u\|_{L^\phi(Q)}$, then we get
\begin{align*}
\int_{Q}\phi_n\left(\frac{|u(y)-u_Q|}{\lambda}\right)dy&\leq
\int_{Q}\phi_n\left(\frac{|u(y)-u_Q|}{\lambda\left[\int_{Q}\phi\left(\frac{C_1|\nabla u(y)|}{\lambda}\right)dy\right]^\frac{1}{n}}\right)dy\\&
\leq\int_{Q}\phi(\frac{C_1|\nabla u(y)|}{\lambda})dy \leq 1
\end{align*}
as desired.
\end{proof}

Recalled the Fefferman-Stein type vector-valued inequality for Hardy-Littlewood maximum operator in Orlicz space. Denote by $\mathcal{M}$ the Hardy-Littlewood maximum operator,
\begin{equation*}
\mathcal{M}(g)(x)=\sup_{x\in Q}\fint_{Q}|g|dx
\end{equation*}
with the supremum taken over all cubes $Q \subset \mathbb{R}^n$ containing $x$.
Next lemmas would give the property of Hardy-Littlewood maximum operator of the Orlicz-Sobolev Space and chain property.

\begin{lemma}[\cite{cmp11}]\label{le2.4}
Let $\psi \in \Delta_2\cap\nabla_2$ be a Young function. For any $0<q<\infty$,  there exists a constant $C>1$ depending on $n,  q,  K_\psi$ and $a$ such that for all sequences $\{f_j\}_{j\in\mathbb{N}}$ , we have
\begin{equation*}
\int_{\mathbb{R}^n}\psi\left(\left[\sum_{j\in\mathbb{N}}(\mathcal{M}(f_j))^2\right]^\frac{1}{q}\right) dx\leq C(n, K_\psi, a)\int_{\mathbb{R}^n}\psi\left(\left[\sum_{j\in\mathbb{N}}(f_j)^2\right]^\frac{1}{q}\right) dx.
\end{equation*}
\end{lemma}
\begin{lemma}\label{le2.5}
For any constant $k\geq1$, sequence $\{a_j\}_{j\in\mathbb{N}}$, and cubes $\{Q_j\}_{j\in\mathbb{N}}$ with $\sum_{j}\chi_{Q_j}\leq k$, we have
\begin{equation*}
\sum_{j}|a_j|\chi_{kQ_j}\leq C(k, n)\sum_j[\mathcal{M}(|a_j|^{\frac{1}{2}}\chi_{Q_j})]^2.
\end{equation*}
\end{lemma}
\begin{proof}
By the definition of $\mathcal{M}$, ~we know
\begin{equation*}
\chi_{kQ_j}\leq k^n\mathcal{M}({\chi}_{Q_j}).
\end{equation*}
So
\begin{equation*}
\sum_{j}|a_j|\chi_{kQ_j}=\sum_{j}(|a_j|^\frac{1}{2}\chi_{kQ_j})^2\leq k^{2n}\sum_j[\mathcal{M}(|a_j|^{\frac{1}{2}}\chi_{Q_j})]^2.
\end{equation*}
\end{proof}

Now we begin to prove Theorem \ref{th1}(i).
\begin{proof}[Proof of Theorem \ref{th1}(i)]
Let $\Omega$ be a c-John domain. 
Applying  Boman   \cite{b82} and Buckley \cite{bkl95}, a John domain  $\Omega$ enjoys the following chain property: 
for every integer $\kappa>1$,  there exist a positive constant $C(\kappa,  \Omega)$ and a collection $\mathcal{F}$ of the cubes such that

(i) $Q\subset \kappa Q \subset \Omega$  for all $Q \in \mathcal{F},  \Omega =\cup_{Q\in \mathcal{F}}Q$ and
\begin{align*}
\sum _{Q\in\mathcal{F}}\chi_{\kappa Q}  \leq  C_{\kappa, c} \chi_{\Omega}. 
\end{align*}

(ii) Fix a cube $Q_{0} \in \mathcal{ F} $. For any other $Q \in  \mathcal{F}$,
there exist a subsequence $\{Q_{j}\}_{j=1}^{N} \subset \mathcal{F}$,  satisfying that $Q = Q_{N} \subset C_{\kappa, c}  Q_{j} $,  $C_{\kappa, c}^{-1}|Q_{j+1}| \leq |Q_{j} | \leq C_{\kappa, c}|Q_{j+1}|$ and
 $|Q_{j} \cap Q_{j+1}|\geq C_{\kappa, c}^{-1} \min\{|Q_{j} |,  |Q_{j+1}|\}$ for all $j =0,  \ldots,  N-1$.

Set $\kappa=5n$.  Since  $Q\subset 5nQ \subset \Omega$ for each $Q\in \mathcal{F}$, we get
\begin{align*}
d(Q, \partial\Omega)\geq d(Q, \partial(5nQ)\geq\frac{5n-1}{2}l(Q)\geq2nl(Q).
\end{align*}
Moreover,
\begin{align*}
|x-y|\leq\sqrt{n}l(Q)\leq nl(Q)\leq \frac{1}{2}d(Q, \partial\Omega)\leq\frac{1}{2}d(x, \partial\Omega), ~\forall x, y\in Q\in\mathcal{F}.
\end{align*}

Let $u\in \dot{W}^{1,\phi}(\Omega)$.
Up to  approximating by $\min\{\max\{u, -N\}, N\}$,  we may assume that $u\in L^\infty(\Omega)$,  and by the boundedness of $\Omega$,   $u\in L^1(\Omega)$.

Using the convexity of $\phi_n$,  we have
\begin{align*}
I:&=\int_{\Omega} \phi_n\left(\frac{|u(z)- u_{\Omega}|}{\lambda}\right)dz\\&
\leq\int_{\Omega} \phi_n\left[\frac{1}{2}\left(\frac{2|u(z)- u_{Q_0}|+2|u_{\Omega}- u_{Q_0}|}{\lambda}\right)\right]dz\\&
\leq \frac{1}{2}\left[\int_{\Omega} \phi_n\left(\frac{2|u(z)- u_{Q_0}|}{\lambda}\right)dz+|\Omega|\phi_n\left(\frac{2|u_{\Omega}- u_{Q_0}|}{\lambda}\right)\right].
\end{align*}
By Jensen inequality,
$$|\Omega|\phi_n\left(\frac{2|u_{\Omega}- u_{Q_0}|}{\lambda}\right)\leq\int_{\Omega} \phi_n\left(\frac{2|u(z)- u_{Q_0}|}{\lambda}\right)dz.$$
Since $\chi_\Omega\leq\sum_{Q\in \mathcal{F}}\chi_Q$ as given (i) above,
\begin{align*}
I&\leq\int_{\Omega} \phi_n\left(\frac{2|u(z)- u_{Q_0}|}{\lambda}\right)dz\\&
\leq\sum_{Q\in\mathcal{F}}\int_Q\phi_n\left(\frac{2|u(z)- u_{Q_0}|}{\lambda}\right)dz\\&
\leq\frac{1}{2}\sum_{Q\in\mathcal{F}}\int_Q\phi_n\left(\frac{4|u(z)- u_{Q}|}{\lambda}\right)dz+\frac{1}{2}\sum_{Q\in\mathcal{F}
\setminus\{Q_0\}}|Q|\phi_n \left(\frac{4|u_{Q}- u_{Q_0}|}{\lambda}\right)\\&:=\frac{1}{2}I_1+\frac{1}{2}I_2.
\end{align*}
Then it suffices to show that
\begin{align*}
I_i \leq \int_\Omega\phi\left(\frac{|\nabla u(x)|}{\lambda/C(n,C_{\kappa,c} , K_{\phi})}\right)dx
\mbox{\quad for $i=1,2$.}
\end{align*}

To bound $I_1$,
for any $Q \in \mathcal{F}$, applying inequality \eqref{eq2.1},
\begin{align*}
I_1\leq\sum_{Q\in\mathcal{F}}\int_Q\phi\left(\frac{|\nabla u(x)|}{\frac{\lambda}{4C_1}}\right)dx.
\end{align*}
Together with $\sum \chi_{Q}(x)\leq C_{\kappa,c} \chi_{\Omega}(x)$ as  in (i) above,
by the convexity we know
\begin{align*}
I_1 \leq C_{\kappa,c} \int_\Omega\phi\left(\frac{|\nabla u(x)|}{\frac{\lambda}{4C_1}}\right)dx
\leq\int_\Omega\phi\left(\frac{|\nabla u(x)|}{\lambda/C(n,C_{\kappa,c} )}\right)dx.
\end{align*}

To estimate $I_2$, for each $Q\in \mathcal{F}$, applying the chain property given in (ii) above,
for any $Q\in\cf$ with $Q\ne Q_0$, we obtain
\begin{align*}
|u_{Q}-u_{Q_0}|& \leq \sum\limits _{j=0}^{N-1}  |u_{Q_j}-u_{Q_{j+1}}|\\
&\leq  \sum\limits _{j=0}^{N-1}(  |u_{Q_j}-u_{Q_{j+1}\cap Q_j}|+|u_{Q_{j+1}}-u_{Q_{j+1}\cap Q_j}|).
\end{align*}
For adjacent cubes $Q_{j}, Q_{j+1}$,
one has $|Q_{j} \cap Q_{j+1}|\geq C_{\kz,c}^{-1} \min\{|Q_{j} |, |Q_{j+1}|\}$
and $C_{\kappa,c}^{-1}|Q_{j+1}| \leq |Q_{j} | \leq C_{\kappa,c}|Q_{j+1}|$.
This implies
\begin{align*}
|u_{Q_{j }}- u_{Q_j \cap Q_{j+1}}|
&\leq  \frac{1}{|Q_j \cap Q_{j+1}|}\int_{Q_j}|u(v)-u_{Q_j}|\, dv\\
&\leq  \frac{ C_{\kz,c}}{ \min\{|Q_{j} |, |Q_{j+1}|\}}\int_{Q_j}|u(v)-u_{Q_j}|\, dv\\
&\leq\frac{C^{2}_{\kz,c}}{|Q_j|} \int_{Q_j}|u(v)-u_{Q_j}|\, dv,
\end{align*}
and  also  similar estimate for $|u_{Q_{j+1 }}- u_{Q_j \cap Q_{j+1}}|$.
Therefore we get
\begin{align*}
|u_{Q}-u_{Q_0}|&
\leq 2 C_{\kz,c}^{2}\sum\limits _{i=1}^{N}  \fint_{Q_j}|u(v)-u_{Q_j}|\,dv.
\end{align*}
For each cube $Q_j$,
by Jessen inequality, one has
\begin{align*}
\fint_{Q_j}\frac{|u(v)-u_{Q_{j}}|}{\lambda}dv&={\phi_n}^{-1}\circ\phi_n\left(\fint_{Q_j}\frac{|u(v)-u_{Q_{j}}|}{\lambda}dv\right)\\&
\leq{\phi_n}^{-1}\left(\fint_{Q_j}\phi_n\left(\frac{|u(v)-u_{Q_{j}}|}{\lambda}\right)dv\right).
\end{align*}
Using inequality \eqref{eq2.1},
\begin{align*}\fint_{Q_j}\frac{|u(v)-u_{Q_{j}}|}{\lambda}dv
\leq{\phi_n}^{-1}\left(\int_{Q_j}\phi\left(\frac{|\nabla u(v)|}{\lambda/C_1}\right)dv\right)
:={\phi_n}^{-1}\left(\int_{Q_j}f(v)dv\right).
\end{align*}
Therefore,
\begin{align*}
\frac{4|u_{Q}- u_{Q_0}|}{\lambda}\leq8{C_{\kappa,c}}^2\sum_{j=0}^{N}{\phi_n}^{-1}\left(\fint_{Q_j}f(v)dv\right).
\end{align*}
Since $\phi_n  \in \triangle_{2}$,
we know  $ \phi_n(tx) \geq t^{K_{\phi_n}-1} \phi_n(x)$ for all $t \in [1,\infty)$ and $x \in \rr$.
Together with Lemma \ref{s2.l4},
we get
\begin{align*}
\phi_n\left(8{C_{\kappa,c}}^2\sum_{j=0}^{N}{\phi_n}^{-1}\left(\fint_{Q_j}f(v)dv\right)\right)\leq C(C_{\kappa,c},K_\phi )\phi_n\left(\sum_{j=0}^{N}{\phi_n}^{-1}\left(\fint_{Q_j}f(v)dv\right)\right).
\end{align*}
Applying   $Q = Q_{N} \subset C_{\kappa,c}  Q_{j} $   given in (ii), one has
$$|Q|\phi_n\left(\sum_{j=0}^{N}{\phi_n}^{-1}\left(\fint_{Q_j}f(v)dv\right)\right)\leq\int_Q\phi_n\left(\sum_{P\in \mathcal{F}}{\phi_n}^{-1}\left(\fint_{P}f(v)dv\right)\chi_{C_{\kappa,c}P}\right)(x)dx.$$
By $\sum _{Q\in\mathcal{F}}\chi_{ Q}\leq\sum _{Q\in\mathcal{F}}\chi_{\kappa Q}  \leq  C_{\kappa,c} \chi_{\Omega}$ as given (i) above,
 \begin{align*}
I_2&\leq C(C_{\kappa,c},K_\phi )\sum_{Q\in \mathcal{F}}\int_Q\phi_n\left(\sum_{P\in \mathcal{F}}{\phi_n}^{-1}\left(\fint_{P}f(v)dv\right)\chi_{C_{\kappa,c}P}\right)(x)dx\\&
\leq C(C_{\kappa,c},K_\phi )\int_{\Omega}\phi_n\left(\sum_{P\in \mathcal{F}}{\phi_n}^{-1}\left(\fint_{P}f(v)dv\right)\chi_{C_{\kappa,c}P}\right)(x)dx.
\end{align*}
Using Lemma \ref{le2.5}, we know
\begin{align*}
I_2\leq C(C_{\kappa,c},K_\phi )\int_{\Omega}\phi_n\left(\sum_{P\in \mathcal{F}}\left\{\mathcal{M}\left[\left({\phi_n}^{-1}
(\fint_{P}f(v)dv)\right)^\frac{1}{2}\chi_{P}\right]\right\}^2\right)(x)dx.
\end{align*}
By Lemma \ref{s2.l4}, we know $\phi_n\in \Delta_2\cap\nabla_2$.
Then $\phi_n(t^2)\in \Delta_2\cap\nabla_2$.
Applying Lemma \ref{le2.4} to $q=2$ and $\psi(t):=\phi_n(t^2)$,  we obtain
\begin{align*}
I_2\leq CC(C_{\kappa,c},K_\phi,a )\int_{\Omega}\phi_n\left(\sum_{P\in \mathcal{F}}\left({\phi_n}^{-1}(\fint_{P}f(v)dv)\right)\chi_{P}\right)(x)dx.
\end{align*}
Let $a_{P}=|P|^{\frac{\beta}{n}}\fint_{P} f(v)\,dv$. For each $x \in \Omega$,
by the increasing property of $\phi_n$ and the convexity of $\phi_n$, we have
 \begin{align*}
\phi_n\left(\sum_{P\in \mathcal{F}}\left({\phi_n}^{-1}(a_P)\right)\chi_{P}(x)\right)
&=\phi_n\left(\frac{\sum_{P\in\mathcal{F}}\chi_{P}(x)}{\sum_{P\in\mathcal{F}}\chi_{P}(x)}\sum_{P\in \mathcal{F}}\left({\phi_n}^{-1}(a_P)\right)\chi_{P}(x)\right)\\&
\leq\phi_n\left(\frac{C_{\kappa,c}}{\sum_{P\in\mathcal{F}}\chi_{P}(x)}\sum_{P\in \mathcal{F}}\left({\phi_n}^{-1}(a_P)\right)\chi_{P}(x)\right)\\&
\leq\sum_{P\in\mathcal{F}}\frac{\chi_P(x)}{\sum_{P\in\mathcal{F}}
\chi_{P}(x)}\phi_n(C_{\kappa,c}{\phi_n}^{-1}(a_P)).
\end{align*}
Applying  $ \phi_n(tx) \geq t^{K_{\phi_n}-1} \phi_n(x)$ for all $t \in [1,\infty)$ and $x \in \rr$.
and $\chi_{\Omega} \leq \sum \chi_{Q} $ as given in (i) above, one gets
 \begin{align*}
\phi_n\left(\sum_{P\in \mathcal{F}}\left({\phi_n}^{-1}(a_P)\right)\chi_{P}(x)\right)
&\leq \sum\limits _{P \in \mathcal{F}} \frac{C(C_{\kappa, c},K_\phi) }{\chi_{\Omega}(x)}\phi_{n}\left(  \phi_{n}^{-1}(a_P) \right)\chi_{P}(x)\\
&\leq C(C_{\kappa, c},K_\phi)\sum_{P\in\mathcal{F}}\chi_P(x)a_P.
\end{align*}
Using $\sum \chi_{Q} \leq C_{\kappa,c} \chi_{\Omega}$ again, one gets
\begin{align*}
I_2&\leq C(C_{\kappa,c},K_\phi,a )\int_\Omega\sum_{P\in\mathcal{F}}a_P\chi_P(x)dx\\&
\leq C(C_{\kappa,c},K_\phi,a )\sum_{P\in\mathcal{F}}a_P|P|
=    C(C_{\kappa,c},K_\phi,a )\sum_{P\in\mathcal{F}}\int_P f(v)dv\\&
\leq C(C_{\kappa,c},K_\phi,a)\int_\Omega\phi\left(\frac{|\nabla u(v)|}{\lambda/C_1}\right)dv.
\end{align*}
By the convexity, one has
\begin{align*}
I_2 \leq\int_\Omega\phi\left(\frac{|\nabla u(v)|}{\lambda/C(n,C_{\kappa,c},K_\phi,a)}\right)dv.
\end{align*}
Combing the estimates $I_1$ and $I_2$, we complete the proof.
\end{proof}

\section{Proof of Theorem 1.2 (ii)}\label{s3}
To prove Theorem \ref{th1} (ii), we need the following estimates and Lemmas which would be proved later.

Let $z\in \Omega,  ~d(z, \partial\Omega)\leq m <\diam\Omega$. Denote $\Omega_{z, m}$ by a component of $\Omega\setminus{\overline {B_{\Omega}(z, m)}}$.
For $t>r\geq m$ with $\Omega_{z, m}\neq\varnothing,  $
define $u_{z, r, t}$ in $\Omega$ as
\begin{align}\label{3.1}
u_{z, r, t}(y)=\left\{\begin{array}{ll}
0,& y\in \Omega\setminus [\Omega_{z, m} \setminus B_\Omega(z, r)]\\
\frac{|y-z|-r}{t-r} , \,\,& y\in \Omega_{z, m}\cap [B(z, t)\setminus B(z, r)] \\
1,& y\in\Omega_{z, m} \setminus B_\Omega(z, t),
\end{array}\right.
\end{align}
where $B_\Omega(z, t)=B(z, t)\cap\Omega$.

It's not difficult to know that
$u_{z, r, t}$ is Lipschitz with the Lipschitz constant  $\frac{1}{t-r}$.

\begin{lemma}\label{s3.l1}
Let $\phi$ be a Young function.
For any bounded domain $\Omega\subset\mathbb{R}^n$ and $z\in \Omega$ with $d(z, \partial\Omega)\leq m<\diam\Omega$.
For $t>r\geq m$,  we have $u_{z, r, t}\in \dot {W}^{1, \phi}(\Omega)$  with
$$\|u_{z, r, t}\|_{\dot {W}^{1, \phi}(\Omega)}\leq \left[{\phi}^{-1}\left(\frac{1}{|\Omega_{z, m}\setminus B(z, r)|}\right)(t-r)\right]^{-1}. $$

\end{lemma}

\begin{proof}
Noting that  $u_{z, r, t}$ is Lipschitz with the Lipschitz constant  $\frac{1}{t-r}$,
then $\nabla u_{z,r,t}$ almost exists and $|\nabla u_{z,r,t}|\leq\frac{1}{t-r}$.
By the definition of $u_{z,r,t}$,
we know $|\nabla u_{z,r,t}|=0$ in $\Omega\setminus [\Omega_{z,m} \setminus B_\Omega(z,r)]$
and $\Omega_{z,m} \setminus B_\Omega(z,t)$.
Hence
\begin{align*}
H:&=\int_\Omega\phi\left(\frac{|\nabla u_{z,r,t}(x)|}{\lambda}\right)dx
\leq\int_{\Omega_{z,m}\setminus {B(z,r)}}\phi\left(\frac{1}{\lambda(t-r)}\right)dx.
\end{align*}
 Letting $\lambda \geq \left[{\phi}^{-1}\left(\frac{1}{|\Omega_{z,m}\setminus B(z,r)|}\right)(t-r)\right]^{-1}$, we have  $H \leq 1$ as desired.
\end{proof}

For $x_0 , z \in \Omega$, let $r>0$ such that $d (z, \partial \Omega) <r < |x_0 -z|$. Define
$$
\omega_{x_0, z, r}(y)= \frac{1}{r}\inf \limits _{\gamma(x_0,y)} \ell (\gamma \cap B(z,r)), \quad \forall y \in \Omega ,
$$
where the infimum is taken over all rectifiable curves $\gamma$ joining $x_0$ and $y$.

\begin{lemma}\label{s3.l2}
Let $\phi$  be a Young function. For any bounded domain $\Omega \subset \rn$, $x_0 , z \in \Omega$ and $r>0$ satisfying $d (z, \partial \Omega) <r < |x_0 -z|$, we have $w_{x_0,z,r} \in \dot{ W }^{1,\phi}(\Omega)$ with
$$\|\omega_{x_0,z,r}\|_{\dot{ W }^{1,\phi}(\Omega)}\le C\left [\phi^{-1}\left(r^{-n}\right)r\right]^{-1}$$
where $C \geq 1$ is depending only on $n, \omega_n$ and $\phi$.
\end{lemma}
\begin{proof}
Let $\gamma _{x,y}$ be the segment joining $x,y$.
Noting that $l(\gamma_{x,y}\cap B(z,r))\leq|x-y|$ for any $x\in \Omega$ and $y\in \Omega$,
together with a curve $\gamma(x_0,x)\cup\gamma_{x,y}$  joining $x_0, x$,
we have
\begin{align*}
\omega_{x_0,z,r}(y)\leq\omega_{x_0,z,r}(x)+\frac{1}{r}|x-y|.
\end{align*}
Similarly, $\omega_{x_0,z,r}(x)\leq\omega_{x_0,z,r}(y)+\frac{1}{r}|x-y|$.
Therefore, we get $|\omega_{x_0,z,r}(y)-\omega_{x_0,z,r}(x)|\leq\frac{1}{r}|x-y|$,
that is, $\omega_{x_0,z,r}$ is Lipschitz and
$\nabla \omega_{x_0,z,r}$ exists with $|\nabla \omega_{x_0,z,r}|\leq\frac{1}{r}$.

Noting that $d(x,\partial\Omega)\leq|x-z|+d(z,\partial\Omega)\leq|x-z|+r$
for $x\in\Omega\setminus B(z,6r), y\in B(x,\frac{1}{2}d(x,\partial\Omega))$, then we know
\begin{align*}
|y-z|\geq|x-z|-|y-x|\geq|x-z|-\frac{1}{2}(|x-z|+r)
=\frac{1}{2}|x-z|-\frac{r}{2}\geq3r-\frac{r}{2}\geq2r.
\end{align*}
that is, $ B(x,\frac{1}{2}d(x,\partial\Omega))\cap B(z,2r)=\varnothing$.
Let $\gamma_{x,y}$ is the segment joining $x,y$.
Then $\gamma_{x,y}$ is in $B(x,\frac{1}{2}d(x,\partial\Omega))$.
Moreover, $\gamma_{x,y}\subset\Omega\setminus\ B(z,r)$.
For any curve $\gamma(x_0,x)$, $\gamma(x_0,x)\cup\gamma_{x,y}$ joining $x_0$ and $y$,
we get
$$l((\gamma(x_0,x)\cup\gamma_{x,y})\cap B(z,r))=l(\gamma(x_0,x)\cap B(z,r)).$$
So $\omega_{x_0,z,r}(y)\leq\omega_{x_0,z,r}(x)$.

Similarly, we have $\omega_{x_0,z,r}(x)\leq\omega_{x_0,z,r}(y)$.
Hence $\omega_{x_0,z,r}(x)=\omega_{x_0,z,r}(y),\forall x\in \Omega\setminus B(z,6r) ,y\in B(x,\frac{1}{2}d(x,\partial\Omega)).$
Since $\omega_{x_0,z,r}(x)=\omega_{x_0,z,r}(y)$ for any $ x\in \Omega\setminus B(z,6r) ,y\in B(x,\frac{1}{2}d(x,\partial\Omega))$
then $|\nabla \omega_{x_0,z,r}(x)|=0$ for any $x\in \Omega\setminus B(z,6r)$.
Hence
\begin{align*}
H:&=\int_\Omega\phi\left(\frac{|\nabla \omega_{x_0,z,r}(x)|}{\lambda}\right)dx
=\int_{\Omega\cap B(z,6r)}\phi\left(\frac{|\nabla \omega_{x_0,z,r}(x)|}{\lambda}\right)dx\\&
\leq\int_{\Omega\cap B(z,6r)}\phi\left(\frac{1}{\lambda r}\right)dx
\leq\omega_n(6r)^n\phi\left(\frac{1}{\lambda r}\right).
\end{align*}
If $\lambda=M\left [\phi^{-1}\left(r^{-n}\right)r\right]^{-1}$ with $M=\omega_n(6r)^n$, then $H \leq 1$.
\end{proof}

\begin{lemma}\label{s3.l3}
Let  $\phi \in \Delta_2$ be a Young function  with $K_{\phi}<2^{n}$ in $\eqref{da2}$ and  a bounded domain $\Omega\subset\mathbb{R}^n$ supports the $(\phi_{n}, \phi)$- Poincar\'e inequality $\eqref{ueq1}$. Fix a point $x_0$ so that $r_0:=\max\{d(x, \partial\Omega):x\in \Omega\}=d(x_0, \partial\Omega)$.  Assume that $x,  x_0\in \Omega\setminus{\overline {B(z, r)}}$ for some $z\in \Omega$ and $r\in(0, 2\diam\Omega)$,  there exists a positive constant $b_0$ that $x, x_0$ are contained in the same component of $\Omega\setminus{\overline {B(z, b_0r)}}$.
\end{lemma}
\begin{proof}
Let $$b_{x,z,r}:= \sup \{ c \in (0,1],  x, x_0  \mbox{\,are contained in the same component of } \Omega \setminus \overline{B(z,cr)}\}.$$
To get $b_0$, it sufficient to prove $b_{x,z,r}$ has a positive lower bound independent of $x,z, r$.
We may assume $b_{x,z,r} \leq \frac{1}{10}$.
Denote $\Omega_{x}$ as the component of $\Omega \setminus \overline{B(z,2b_{x,z,r}r)}$ containing $x$.
If exists a constant $C \geq 1$ independent of $x,z,r$ such that
\begin{align}\label{w2.1}
\frac{r}{C} (\frac{1}{2}-2 b_{x,z,r} ) \leq |\Omega_{x}|^{\frac{1}{n}} \leq C2 b_{x,z,r}r,
\end{align}
we know $b_{x,z,r} > \frac{1}{4(C^{2}+1)}$ as desired.
Set $c_0 = 2 b_{x,z,r} < \frac{1}{5}$. Denote by  $\Omega_{x_0}$  the component of $\Omega \setminus \overline{B(z,c_0r)}$ containing $x_0$. Observing
 $$
 r_0 \leq \max \limits _{y \in B(z,c_0r)} |x_0 -y| \leq r+c_0 r +d(x_0 , B(z,r)) \leq \frac{6}{5}r+ d(x_0 , B(z,r))$$
and
$$
d(x_0 , B(z,c_0r)) > |x_0- z |- \frac{1}{5} = d(x_0 , B(z,r))+r-\frac{1}{5}r =d(x_0 , B(z,r))+\frac{4}{5}r,
$$
we obtain
$d(x_0 , B(z,c_0r)) \geq \frac{r_0}{2}$, hence
\begin{align}\label{w2.2}
B(x_0 , \frac{r_0}{2}) \subset \Omega_{x_0} \subset \Omega \setminus \Omega_{x}.
\end{align}
Define
$$
w(y)= \frac{1}{c_0 r}\inf \limits _{\gamma(x_0,y)} \ell (\gamma \cap B(z,c_0r)), \quad \forall y \in \Omega,
$$
where the infimum is taken over all rectifiable curves $\gamma$ joining $x_0$ and $y$.

By Lemma \ref{s3.l2},
$$\|\omega\|_{{\dot {W}}^{1,\phi}(\Omega)}\leq C\left[{\phi}^{-1}\left(\frac{1}{(c_0r)^n}\right)c_0r\right]^{-1},$$
together with $(\phi_n,\phi)$-Poincar\'e inequality \eqref{ueq1}, we know
\begin{align*}
\|\omega-\omega_\Omega\|_{L^{\phi_n}(\Omega)}\leq C\|\omega\|_{{\dot {W}}^{1,\phi}(\Omega)}\leq C\left[{\phi}^{-1}\left(\frac{1}{(c_0r)^n}\right)\right]^{-1}\frac{1}{c_0r}.
\end{align*}
On the other hand, by \eqref{w2.2},  $y\in B(x_0, \frac{1}{2}r_0),  \omega(y)=0$.
Since $\Omega$ is bounded,  $r_0>0,$ we have $\frac{|\diam\Omega|}{r_0^n}\leq C$.
Using the convexity of $\phi_{n}$,
\begin{align*}
\int_\Omega\phi_{n}\left(\frac{|\omega(x)|}{\lambda}\right)dx\leq\frac{1}{2}\int_\Omega\phi_{n}\left(\frac{|\omega(x)-\omega_\Omega|}{\lambda}\right)dx+\frac{|\Omega|}{2}\phi_{n}\left(\frac{|\omega_{B(x_0,\frac{1}{2}r_0)}-\omega_\Omega|}{\lambda}\right).
\end{align*}
By the Jensen inequality,
\begin{align*}
|\Omega|\phi_{n}\left(\frac{|\omega_{B(x_0,\frac{1}{2}r_0)}-\omega_\Omega|}{\lambda}\right)&
\leq|\Omega|\fint_{B(x_0,\frac{1}{2}r_0)}\phi_{n}\left(\frac{|\omega(x)-\omega_\Omega|}{\lambda}\right)dx\\&
\leq\frac{|\Omega|}{|B(x_0,\frac{1}{2}r_0)|}\int_\Omega\phi_{n}\left(\frac{|\omega(x)-\omega_\Omega|}{\lambda}\right)dx\\&
\leq2^nC^n\int_\Omega\phi_{n}\left(\frac{|\omega(x)-\omega_\Omega|}{\lambda}\right)dx.
\end{align*}
Hence
$$\int_\Omega\phi_{n}\left(\frac{|\omega(x)|}{\lambda}\right)dx\leq C\int_\Omega\phi_{n}\left(\frac{|\omega(x)-\omega_\Omega|}{\lambda}\right)dx,$$
furthermore, we get
\begin{equation}\label{t4.23}
\|\omega\|_{L^{\phi_n}(\Omega)}\leq C\|\omega-\omega_\Omega\|_{L^{\phi_n}(\Omega)}.
\end{equation}

Since for any $ y\in \Omega_x, ~\omega(y)\geq1$,
\begin{align*}
\int_\Omega\phi_{n}\left(\frac{|\omega(x)|}{\lambda}\right)dx
\geq\phi_{n}\left(\frac{1}{\lambda}\right)|\Omega_x|,
\end{align*}
then we know
\begin{align*}
\|\omega\|_{L^{\phi_n}(\Omega)}\geq \left[{\phi_{n}}^{-1}\left(\frac{1}{|\Omega_x|}\right)\right]^{-1}.
\end{align*}
Therefore,
\begin{align*}
C\phi^{-1}\left[\frac{1}{(c_0r)^n}\right](c_0r)\leq{\phi_{n}}^{-1}\left[\frac{1}{|\Omega_x|}\right].
\end{align*}
By $\frac{H(A)}{A}\leq C\frac{1}{{\phi(A)}^\frac{1}{n}}$ in \eqref{te1},
letting $A=\phi^{-1}\left[\frac{1}{(c_0r)^n}\right]$, we have
\begin{align*}
\frac{{\phi_n}^{-1}\left[\frac{1}{(c_0r)^n}\right]}{\phi^{-1}\left[\frac{1}{(c_0r)^n}\right]}\leq C(c_0r),
\end{align*}
that is,
\begin{align*}
{\phi_n}^{-1}\left[\frac{1}{(c_0r)^n}\right]\leq C{\phi_{n}}^{-1}\left[\frac{1}{|\Omega_x|}\right].
\end{align*}
By Lemma \ref{s2.l4}, $\phi_n\in \Delta_2$,
and the fact $ \phi_n(tx) \geq t^{K_{\phi_n}-1} \phi_n(x)$ for all $t \in [1,\infty)$ and $x \in \rr$,
we have $\frac{1}{(c_0r)^n}\leq C\frac{1}{|\Omega_x|}$ and
\begin{equation}\label{4.233}
{|\Omega_x|}^\frac{1}{n}\leq C(c_0r).
\end{equation}
For $j\geq0$ with $\Omega_x\setminus \overline{B(z, c_jr)}\neq\varnothing$,
define $v_j$ in $\Omega$ as
\begin{align*}
v_j(y)=\left\{\begin{array}{ll}
0& y\in \Omega\setminus [\Omega_x \setminus B_\Omega(z, c_{j+1}r)]\\
\frac{|y-z|-c_jr}{c_{j+1}r-c_jr} & y\in \Omega_x\cap [B(z, c_jr)\setminus B(z, c_{j+1}r)], \\
1& y\in\Omega_x \setminus B_\Omega(z, c_jr),
\end{array}\right.
\end{align*}
Let $\Omega_{z, x}=\Omega_x, r=c_jr$ and $t=c_{j+1}r$,
then $v_j(y)=u_{z, c_jr, c_{j+1}r}(y)$ where $u_{z, c_jr, c_{j+1}r}(y)$ is defined in \eqref{3.1}.
Applying Lemma \ref{s3.l1},  we have
\begin{align*}
\|v_j\|_{{\dot {W}}^{1,\phi}(\Omega)}\leq C\left[{\phi}^{-1}\left(\frac{1}{|\Omega_x\setminus B(z,c_jr)|}\right)(c_{j+1}r-c_jr)\right]^{-1}.
\end{align*}
Applying \eqref{w2.1}, we have $v_j(y)=0$ for $y\in B(x_0, \frac{1}{2}r_0)$.
Similarly to \eqref{t4.23},  we get
\begin{equation}
\|v_j\|_{L^{\phi_n}(\Omega)}\leq C\|v_j-{v_j}_\Omega\|_{L^{\phi_n}(\Omega)}.
\end{equation}
And $v_j(y)=1$ for $y\in \Omega_x\setminus {B_\Omega(z, c_jr)}$,  then we have
\begin{align*}
\|v_j\|_{L^{\phi_n}(\Omega)}\geq \left[{\phi_{n}}^{-1}\left(\frac{1}{|\Omega_x\setminus {B_\Omega(z,c_jr)}|}\right)\right]^{-1}.
\end{align*}
By the $(\phi_n, \phi)$-Poincar\'e inequality \eqref{ueq1}, we know
\begin{align*}
{\phi_{n}}^{-1}\left(\frac{1}{|\Omega_x\setminus {B_\Omega(z,c_jr)}|}\right)\geq C{\phi}^{-1}\left(\frac{1}{|\Omega_x\setminus B(z,c_jr)|}\right)(c_{j+1}r-c_jr).
\end{align*}
By $\frac{H(A)}{A}\leq\frac{ C}{{\phi(A)}^\frac{1}{n}}$ in \eqref{te1},
letting $A={\phi}^{-1}\left(\frac{1}{|\Omega_x\setminus B(z,c_jr)|}\right)$, we get
\begin{align*}
c_{j+1}r-c_jr\leq C|\Omega_x\setminus B(z,c_jr)|^\frac{1}{n}.
\end{align*}
Hence $c_{j+1}-c_jr\leq C|\Omega_x\setminus B(z,c_jr)|^\frac{1}{n}\leq C2^{-\frac{j}{n}}|\Omega_x|^\frac{1}{n}$.

Now we prove that $\sup\left\{c_j\right\}>1$.
Otherwise, we have $ c_j\leq1$ for all $j$.
By $x\in\Omega\setminus\overline{B(x, r)}$,
then there exists $ \delta>0$ such that
$$B(x, \delta)\subset\Omega\setminus\overline{B(x, r)}\subset\Omega\setminus\overline{B(x, c_0r)}.$$
By the connectivity of the $B(x, \delta)$,  we have $B(x, \delta)\subset\Omega_x$.
Then $$B(x, \delta)\subset\Omega_x\setminus\overline{B(x, r)}\subset\Omega_x\setminus B(x, c_jr), $$
and
$$0<|B(x, \delta)|\leq|\Omega_x\setminus\overline{B(x, r)}|\leq|\Omega_x\setminus B(x, c_jr)|=2^{-j}|\Omega_x|.$$
Letting $j\to \infty$ we get a contradiction,
hence $\sup\left\{c_j\right\}>1$.
So there exists $c_j$ such that $c_j\geq\frac{1}{2}.$
Let $j_0=\inf\left\{j\geq1:c_j\leq\frac{1}{2}\right\} $,
then
\begin{align*}
\left(\frac{1}{2}-c_0\right)r\leq(c_{j_0}-c_0)r=\sum_{j=0}^{j_0-1}(c_{j+1}-c_j)r\leq C\sum_{j=0}^{j_0-1}2^{-\frac{j}{n}}|\Omega_x|^\frac{1}{n}\leq 2C|\Omega_x|^\frac{1}{n}.
\end{align*}
So $\frac{r}{C}\left(\frac{1}{2}-2b_{x, z, r}\right)\leq|\Omega_x|^\frac{1}{n}$.
By the \eqref{4.233},  we have
\begin{align*}
\frac{r}{C}\left(\frac{1}{2}-2b_{x, z, r}\right)\leq|\Omega_x|^\frac{1}{n}\leq C2b_{x, z, r}r,  \, C\geq 1.
\end{align*}
Then $b_{x, z, r}\geq\frac{1}{4(C^2+1)}$,
which implies $b>0$.
\end{proof}

\begin{lemma}\label{le421}
Let $s\in(0, 1)$ and $\phi \in \Delta_2$ be a Young function with $K_\phi<2^{n}$ in \eqref{da2},  a bounded domain $\Omega\subset\mathbb{R}^n$ supports the $(\phi_n, \phi)$- Poincar\'e inequality \eqref{ueq1}, then the $\Omega$ has the LLC(2) property,  that is,  there exists a constant $b\in(0, 1)$ such that for all $z\in \mathbb{R}^n$ and $r>0, $ any pair of point in $\Omega\setminus\overline{B(z, r)}$ can be joined in $\Omega\setminus\overline{B(z, br)}$.
\end{lemma}
\begin{proof}
Fix $x_0$ so that $r_0:=\max(d(x, \partial\Omega):x\in \Omega)=d(x_0, \partial\Omega))$
and $b_0$ is the constant in Lemma \ref{s3.l3}.
Then we spilt into three cases to prove it.

Case 1. For $z\notin B\left(x_0, \frac{r_0}{8\diam\Omega}r\right)$, we consider the radius $r$.

If $r>\frac{16(\diam\Omega)^2}{r_0}$, then $\forall y \in\overline{ B\left(z, \frac{r_0}{16\diam\Omega}r\right)}$, we have
$$|y-x_0|\geq|z-x_0|-|z-y|\geq\frac{r_0}{16\diam\Omega}r>\diam\Omega.$$
By $\Omega\subset B(x_0, \diam\Omega)$, we get $\Omega\cap\overline{ B\left(z, \frac{r_0}{16\diam\Omega}r\right)}=\varnothing$. Here, any pair of point in $\Omega\setminus\overline{B(z, r)}$ can be joined in $\Omega\setminus\overline{B(z, \frac{r_0}{16\diam\Omega}r)}=\Omega$.

If $r\leq\frac{16(\diam\Omega)^2}{r_0}$ and $d(z, \partial\Omega)>\frac{b_0r_0}{32\diam\Omega}r$.
When $z\notin\Omega$, then any pair of point in $\Omega\setminus\overline{B(z, r)}$ can be joined in $\Omega\setminus\overline{B\left(z, \frac{b_0r_0}{32\diam\Omega}r\right)}=\Omega$.
When $z\in\Omega$, then $B\left(z, \frac{b_0r_0}{64\diam\Omega}r\right)\subset B\left(z, \frac{b_0r_0}{32\diam\Omega}r\right)\subset \Omega$.
Similar to the process of proving $b_{x, z, r}>0$ in Lemma \ref{s3.l3},  we know $\Omega\setminus\overline{B\left(z, \frac{b_0r_0}{64\diam\Omega}r\right)}$ is a connected set.
Here,  any pair of point in $\Omega\setminus\overline{B(z, r)}$ can be joined in $\Omega\setminus\overline{B\left(z, \frac{b_0r_0}{64\diam\Omega}r\right)}$.

If $r\leq\frac{16(\diam\Omega)^2}{r_0}$ and $d(z, \partial\Omega)\leq\frac{b_0r_0}{32\diam\Omega}r$.
Let $y\in B\left(z, \frac{b_0r_0}{16\diam\Omega}r\right)\cap\Omega$.
By $B\left(y, (1-\frac{b_0}{2})\frac{r_0}{8\diam\Omega}r\right)\subset B\left(z, \frac{r_0}{8\diam\Omega}r\right)\subset B(z, r)$, we know
\begin{align*}
\forall x\in \Omega\setminus \overline {B(z, r)}, x, x_0\in\Omega\setminus \overline {B\left(y, (1-\frac{b_0}{2})\frac{r_0}{8\diam\Omega}r\right)}.
\end{align*}
By Lemma \ref{s3.l3}, $x, x_0$ are in the same component of 
$\Omega\setminus{\overline {B\left(y, b_0(1-\frac{b_0}{2})\frac{r_0}{8\diam\Omega}r\right)}}$.
For any  $ w\in B\left(z, \frac{b_0(1-b_0)r_0}{16\diam\Omega}r\right) $, we have
\begin{align*}
|w-y|\leq|w-z|+|z-y|<\frac{b_0(1-b_0)r_0}{16\diam\Omega}r+\frac{b_0r_0}{16\diam\Omega}r
=b_0\left(1-\frac{b_0}{2}\right)\frac{r_0}{8\diam\Omega}r.
\end{align*}
Then 
\begin{align*}
B\left(z, \frac{b_0(1-b_0)r_0}{16\diam\Omega}r\right)\subset B\left(y, b_0\left(1-\frac{b_0}{2}\right)\frac{r_0}{8\diam\Omega}r\right), 
\end{align*}
and $\Omega\setminus{\overline { B\left(y, b_0\left(1-\frac{b_0}{2}\right)\frac{r_0}{8\diam\Omega}r\right)}}\subset\Omega\setminus\overline {B\left(z, \frac{b_0(1-b_0)r_0}{16\diam\Omega}r\right)}$.
Here,  any pair of point in $\Omega\setminus\overline{B(z, r)}$  can be joined in 
$\Omega\setminus\overline {B\left(z, \frac{b_0(1-b_0)r_0}{16\diam\Omega}r\right)}$.

Case 2. If $z\in B\left(x_0, \frac{r_0}{8\diam\Omega}r\right)$,
for any $ x\in\Omega\setminus\overline{B(z, r)}$,
\begin{align*}
r-\frac{r_0}{8\diam\Omega}r\leq|x-z|-|x_0-z|\leq|x-x_0|\leq \diam \Omega, 
\end{align*}
 we obtain
\begin{align*}
 r\leq\frac{\diam\Omega}{1-\frac{r_0}{8\diam\Omega}}\leq 2\diam \Omega.
\end{align*}
Then
\begin{align*}
 B\left(z, \frac{r_0}{8\diam\Omega}r\right)\subset B\left(x_0, \frac{r_0}{4\diam\Omega}r\right)\subset B\left(x_0, \frac{r_0}{2}\right)\subset B(x_0, r_0)\subset\Omega
\end{align*} 
Similar to the process of proving $b_{x, z, r}>0$ in Lemma \ref{s3.l3},  we have 
$\Omega\setminus\overline{B\left(z, \frac{r_0}{8\diam\Omega}r\right)}$ is a connected set. 
And by
\begin{align*}
\Omega\setminus\overline{B(z, r)}\subset\Omega\setminus\overline{B\left(z, \frac{r_0}{8\diam\Omega}r\right)},
\end{align*}
we know any pair of point in $\Omega\setminus\overline{B(z, r)}$ can be joined in 
$\Omega\setminus\overline{B\left(z, \frac{r_0}{8\diam\Omega}r\right)}$.

Combining above cases,
we get the desired result with $b=\min\left\{\frac{r_0}{16\diam\Omega}, \frac{b_0r_0}{64\diam\Omega}, \frac{b_0(1-b_0)r_0}{16\diam\Omega}  \right\}$.
\end{proof}

\begin{proof}[Proof of Theorem \ref{th1}(ii)]
Let $\Omega\subset\mathbb{R}^n$ be a simply connected planar domain,  or a bounded domain that is quasiconformally equivalent to some uniform domain when $n\leq3$. Assume $\Omega$ supports the $(\phi_\frac{n}{s}, \phi)$-Poincar\'e inequality.

By \cite{bk95,  bsk96}, $\Omega$ has a separation property with $x_0\in\Omega$ and some constant $C_0\geq1$, that is,
 for any $ x\in \Omega$, there exists a curve $\gamma:[0, 1]\to\Omega$, with $\gamma(0)=x, \gamma(1)=x_0$  
 and for any $ t\in[0, 1] $, either $\gamma([0, 1])\subset\overline B:=\overline{B(\gamma(t), C_0d(\gamma(t), \Omega^{\complement}))}$, 
 or for any $ y \in \gamma([0, 1])\setminus\overline B$ belongs to  the different component of $\Omega\setminus\overline B$ .
For any $x\in \Omega$ , let $\gamma$ be a curve as above.
Applying the arguments in \cite{m79}, It suffices to prove
there exists a constant $C>0$ so that
\begin{equation}\label{1.2343}
d(\gamma(t), \Omega^{\complement})\geq C\diam~\gamma([0, t]), ~\forall t\in [0, 1].
\end{equation}
Indeed, \eqref{1.2343} could modify $\gamma$ to get a John curve for $x$.

Using Lemma \ref{le421}, $\Omega$ has the LLC(2) property. Let $a=2+\frac{C_0}{b}$, where $b$ is the constant in Lemma \ref{le421}.

For $t\in[0, 1]$. We split into two cases.

Case 1.  If $d(\gamma(t), \Omega^{\complement})\geq\frac{d(x_0, \Omega^{\complement})}{a}$, then
\begin{align*}
\gamma([0, t])\subset\Omega\subset B\left(\gamma(t), \frac{ad(\gamma(t), \Omega^{\complement})}{d(x_0, \Omega^{\complement})}\diam \Omega\right).
\end{align*}
So 
\begin{align*}
\diam \gamma([0, t])\leq\frac{2ad(\gamma(t), \Omega^{\complement})}{d(x_0, \Omega^{\complement})}\diam \Omega.
\end{align*}
and
\begin{align*}
d(\gamma(t), \Omega^{\complement})\geq\frac{d(x_0, \Omega^{\complement})}{2a\diam\Omega}\diam~\gamma([0, t]).
\end{align*}

Case 2. If $d(\gamma(t), \Omega^{\complement})<\frac{d(x_0, \Omega^{\complement})}{a}$,  we claim that
\begin{align*}
\gamma([0, t])\subset\overline{B\left(\gamma(t), (a-1)d(\gamma(t), \Omega^{\complement})\right)}.
\end{align*}
Otherwise,  there exists
$y\in\gamma([0, t])\setminus\overline{B\left(\gamma(t), (a-1)d(\gamma(t), \Omega^{\complement})\right)}$.
Moreover,
\begin{align*}
|x_0-\gamma(t)|\geq d(x_0, \Omega^{\complement})-d(\gamma(t), \Omega^{\complement})>(a-1)d(\gamma(t), \Omega^{\complement}), 
\end{align*}
we know $x_0, y\in\Omega\setminus\overline{B\left(\gamma(t), (a-1)d(\gamma(t), \Omega^{\complement}\right)}$.
Using Lemma \ref{le421}, $x_0$ and $y$ are contained in the same complement of $\Omega\setminus\overline{B\left(\gamma(t), b(a-1)d(\gamma(t), \Omega^{\complement}\right)}$.
Since $b(a-1)\geq C_0$, then $x_0$ and $y$ are contained in the same complement of $\Omega\setminus\overline{B\left(\gamma(t), C_0d(\gamma(t), \Omega^{\complement}\right)}$, 
which is in contradiction with the separation property.
Hence
\begin{align*}
\gamma([0, t])\subset\overline{B\left(\gamma(t), (a-1)d(\gamma(t), \Omega^{\complement})\right)},
\end{align*}
then 
\begin{align*}
\diam\,\gamma([0, t])\leq2(a-1)d(\gamma(t), \Omega^{\complement}).
\end{align*}
So
\begin{align*}
d(\gamma(t), \Omega^{\complement})\geq\frac{1}{2(a-1)}\diam~\gamma([0, t]).
\end{align*}
Let $C=\min\left\{\frac{d(x_0, \Omega^{\complement})}{2a\diam\Omega}, \frac{1}{2(a-1)}\right\}$, then \eqref{1.2343}  holds. The proof is completed.

\end{proof}

\medskip
 \noindent {\bf Acknowledgment}. The authors would like to thank  Professor Yuan Zhou for several valuable discussions of this paper.
 The authors contributed equally to this work.
The second author are partially supported by National Natural Science Foundation of China (No. 12201238), GuangDong Basic and Applied Basic Research Foundation (Grant No. 2022A1515111056), Bureau of Science and Technology of Huizhou Municipality (Grant No. 2023EQ050040) and the Professorial and Doctoral Scientific Research Foundation of Huizhou University (Grant No. 2021JB035).

\end{document}